\documentclass[12pt, reqno]{amsart}
\usepackage{amssymb,amsmath}
\usepackage{amsmath, amsthm, amscd, amsfonts, amssymb, graphicx, color}
\usepackage{amsmath, amsthm, amssymb}
\usepackage{amsfonts}
\usepackage{graphicx}

\newtheorem{theorem}{Theorem}[section]
\newtheorem{lemma}{Lemma}[section]
\newtheorem{corollary}{Corollary}[section]

\newtheorem{proposition}{Proposition}[section]

\newcommand{\M}{\mathbb{M}_{n}}

\begin{document}

\title[A note on Positive Partial Transpose Blocks]{A note on Positive Partial Transpose Blocks}

\author[M. Alakhrass]{Mohammad Alakhrass}

\address{ Department of Mathematics, University of Sharjah , Sharjah 27272, UAE.}
\email{\textcolor[rgb]{0.00,0.00,0.84}{malakhrass@sharjah.ac.ae}}



\begin{abstract}
In this article, we study the class of PPT blocks. We introduce several inequalities, related to this class,
with emphasis on comparing the main diagonal and the off-diagonal components of a $2 \times 2$ PPT block.
\end{abstract}

\maketitle

\section{Introduction}
Let $\M$ be the algebra of all $n \times n$ complex matrices. For $X \in \mathbb{M}_n$, the notation $X\geq 0$ (resp. $X>0$) will be used to mean that $X$ is positive semidefinite (resp. positive definite). If $X, Y \in \mathbb{M}_n$ are  two Hermitian matrices in $\mathbb{M}_n$,
we write $X \leq Y$ to mean $Y-X \geq 0$. The unitarily invariant norm of $X \in \M$ is denoted by $ \| X \|$. Recall that a norm $\|\cdot\|$ on $\mathbb{M}_n$ is said to be unitarily invariant  if it satisfies the property $\|UXV\|$ for all $X\in\mathbb{M}_n$ and all unitaries $U,V\in\mathbb{M}_n.$

Let $A, B, X \in \M$. Throughout this note, we consider the $2 \times 2$ block matrix $H$ in the form
$$
H=
\left(
\begin{array}{cc}
A & X \\
X^* & B \\
\end{array}
\right)
.$$
It is well known that $H$ positive if and only if the Schur complement of $A$ in $H$ is positive semidefinite provided that $A$ is strictly positive.
That is, $H \geq 0$ if and only if
\begin{equation}\label{Schure criterion}
H /A= B- X^* A^{-1} X \geq 0.
\end{equation}

The $2\times 2$ blocks play an important role in studying matrices and positive matrices in particular. Bhatia book \cite{B-Book 2-2007} provides
a comprehensive survey about block matrices. Furthermore, a positive $2\times 2$ block can be a very useful tool in studying sectorial matrices, see for example \cite{Alakhrass-2021}, \cite{Alakhrass-2020} and \cite{Alakhrass-2019}.

The the partial transpose of the block $H$ is defined by
$$
H^{\tau}=
\left(
\begin{array}{cc}
A & X^* \\
X & B \\
\end{array}
\right).
$$
It is quite clear that the positivity of $H$ does not, in general, imply the positivity of $ H^{\tau}$.
The block $H$ is said to be positive partial transpose, or PPT for short, if both $H$ and $H^{\tau}$ are positive semidefinite. The Schure criterion for positivity implies that $H$ is PPT if and only if
$$
B- X^* A^{-1} X \geq 0 \quad \text{and} \quad B- X A^{-1} X^* \geq 0,
$$
provided that $A$ is strictly positive.

The  PPT criterion ( also called Peres–Horodecki criterion ) plays an important roll in the quantum information theory. For example, PPT
condition is a necessary condition for a mixed quantum state to be separable. Moreover, in low dimensional composite spaces ( two and three) this condition (PPT) is also sufficient. See \cite{HORODECKI-1996} and \cite{Peres-1996}. \\

The class of PPT matrices possess many interesting properties. Therefore, it has attracted a huge interest. See
\cite{T-Ando, Choi-2018, Choi-2017, Hiroshima-2003, Lee-2015, Migh-2015,  Migh-Hiroshima-2015}.

Given a PPT block
$
H=
\left(
\begin{array}{cc}
A & X \\
X^* & B \\
\end{array}
\right).
$
It is well known that for any unitarily invariant norm

\begin{equation}\label{Hiroshima-1}
|| H || \leq || A+B ||.
\end{equation}

This result due to Hiroshima \cite{Hiroshima-2003}. See also \cite{Migh-Hiroshima-2015} for related results.
A stronger result can be obtained when $X$ is Hermitian. In fact it was shown in \cite{Bourin-2013} that if $X$ is Hermitian, then

\begin{equation}\label{Hiroshima-2}
H=
\left(
\begin{array}{cc}
A & X \\
X & B \\
\end{array}
\right)
=\frac{1}{2} \left[U \left(A+B \right)U^* + V \left(A+B \right) V^*   \right],
\end{equation}
for some isometries $U, V \in \mathbb{M}_{2n \times n}$. Here $\mathbb{M}_{2n \times n}$ is the space of all $2n \times n$ complex
matrices.
The inequality \eqref{Hiroshima-1} and the
identity \eqref{Hiroshima-2} present connections between the whole block $H$
and its main diagonal components $A$ and $B$. Recently, some interesting inequalities, connecting the main and the off-diagonal of the PPT Block $H$, have been established. For example, in \cite{Migh-2015}, Lin proved that if $H$ is PPT,  then
\begin{equation}\label{Ming-tr}
tr \left( X^*X \right) \leq tr \left( AB \right)
\end{equation}

In the sense of Loewner, it has been proved, in \cite{Lee-2015}, that for some unitary $U \in \M$
\begin{equation}\label{Lee-Gmean}
|X| \leq  \frac{A \# B + U^*(A \# B)U }{2}.
\end{equation}

An improvement of the inequality \eqref{Lee-Gmean}  was given in \cite{Pan-Fu-2021}, the authors proved that if $H$ is PPT, then
\begin{equation}\label{Pan-Fu}
|X| \leq \left ( A \# B \right)  \#  U^*(A \# B)U,
\end{equation}
for some unitary $U \in \M$.

In this paper, we show that if $H$ is PPT and $t \in [0,1]$, then
\begin{align}\label{MainRR}
|X| &\leq \left ( A \#_t B \right)  \#  U^*(A \#_{1-t} B)U \notag \\
& \leq \frac{A \#_t B + U^*(A \#_{1-t} B)U }{2}.
\end{align}
for some unitary $U \in \M$.
Then we present several consequences of \eqref{MainRR} including inequalities such as \eqref{Ming-tr}, \eqref{Lee-Gmean} and \eqref{Pan-Fu}.
Finally, we present some inequalities that connect the diagonal components to the real part of the off-diagonal components of H.

\section{Preliminaries}
In this section we present some basic properties of positive and PPT blocks. These properties are summarized in
Proportion \ref{Moh-Nov}, \ref{Moh-Nov-2} and \ref{Moh-Nov-3}. To make this note self-contained, we outline the proofs of these propositions. We also recall some important facts about weighted geometric mean of two positive matrices.

\begin{proposition}\label{Moh-Nov}
If
$
H=
\left(
\begin{array}{cc}
A & X \\
X^* & B \\
\end{array}
\right) \geq 0,
$
then
\begin{enumerate}

\item
$
\left(
\begin{array}{cc}
A & -X \\
-X^* & B \\
\end{array}
\right) \geq 0
$
and
$
\left(
\begin{array}{cc}
B & X^* \\
X & A \\
\end{array}
\right) \geq 0.
$

\item
$
\left(
\begin{array}{cc}
 0& X \\
X^* & 0 \\
\end{array}
\right)
\leq
\frac{1}{2}
H.
$
\end{enumerate}
\end{proposition}

\begin{proof}
To see the first part, observe that
$$
\left(
\begin{array}{cc}
A & -X \\
-X^* & B \\
\end{array}
\right)
=
\left(
\begin{array}{cc}
-I & 0 \\
0 & I \\
\end{array}
\right)
\left(
\begin{array}{cc}
A & X \\
X^* & B \\
\end{array}
\right)
\left(
\begin{array}{cc}
-I & 0 \\
0 & I \\
\end{array}
\right)
\geq 0
$$
and
$$
\left(
\begin{array}{cc}
B & X^* \\
X &A  \\
\end{array}
\right)
=
\left(
\begin{array}{cc}
0& I \\
I &0 \\
\end{array}
\right)
\left(
\begin{array}{cc}
A & X^* \\
X & B \\
\end{array}
\right)
\left(
\begin{array}{cc}
0 & I \\
I & 0 \\
\end{array}
\right)
\geq 0.
$$
For the second part, notice that
$$
\frac{1}{2}
H
-
\left(
\begin{array}{cc}
 0& X \\
X^* & 0 \\
\end{array}
\right)
=
\frac{1}{2}
\left(
\begin{array}{cc}
A & -X \\
-X^* & B \\
\end{array}
\right)
\geq 0.
$$
\end{proof}

\begin{proposition}\label{Moh-Nov-2}
If
$
\left(
\begin{array}{cc}
A & X \\
X^* & B \\
\end{array}
\right)
$
is PPT, then the following blocks are positive semidefinite. \\
$
\left(
\begin{array}{cc}
A & \mp X \\
\mp X^* & B \\
\end{array}
\right)
,
\left(
\begin{array}{cc}
A & \mp X^* \\
\mp X & B \\
\end{array}
\right)
,
\left(
\begin{array}{cc}
B & \mp X^* \\
\mp X & A \\
\end{array}
\right),
\left(
\begin{array}{cc}
B & \mp X \\
\mp X^* & A \\
\end{array}
\right).
$
\end{proposition}

\begin{proof}
The semi positivity of the first two blocks follows from the definition of PPT and the first part of Proposition \ref{Moh-Nov}.
The semi positivity of the second two blocks results from conjugating the first two blocks by the unitary
$
\left(
\begin{array}{cc}
0 & I \\
I & 0 \\
\end{array}
\right).
$
\end{proof}

\begin{proposition}\label{Moh-Nov-3}
Let
$
H=
\left(
\begin{array}{cc}
A & X \\
X^* & B \\
\end{array}
\right)
$
be PPT. Then
$$
\left(
\begin{array}{cc}
A & e^{i \theta}X \\
e^{ - i \theta}X^* & B \\
\end{array}
\right)
\quad
\text{and}
\quad
\left(
\begin{array}{cc}
\frac{A+B}{2}  & X \\
X^* & \frac{A+B}{2} \\
\end{array}
\right)
\quad
\text{are PPT.}
\quad
$$
\end{proposition}

\begin{proof}
Let
$
W=
\left(
\begin{array}{cc}
e^{i \theta}  I & 0 \\
0 & I \\
\end{array}
\right).
$
Notice that
$$
\left(
\begin{array}{cc}
A & e^{i \theta}X \\
e^{ - i \theta}X^* & B \\
\end{array}
\right)
= W
\left(
\begin{array}{cc}
A & X \\
X^* & B \\
\end{array}
\right)
W^*
\geq 0,
$$
and
$$
\left(
\begin{array}{cc}
A & e^{- i \theta}X^* \\
e^{i \theta}X & B \\
\end{array}
\right)
= W^*
\left(
\begin{array}{cc}
A & X^* \\
X & B \\
\end{array}
\right)
W
\geq 0.
$$
This implies that the first block is PPT.

Since
$
\left(
\begin{array}{cc}
A & X \\
X^* & B \\
\end{array}
\right)
$
is PPT, Proposition \ref{Moh-Nov-2} implies that
$
\left(
\begin{array}{cc}
B & -X \\
-X^* & A \\
\end{array}
\right)
\geq 0.
$
Therefore,
$$
H=
\left(
\begin{array}{cc}
A & X \\
X^* & B \\
\end{array}
\right)
\leq
\left(
\begin{array}{cc}
A & X \\
X^* & B \\
\end{array}
\right)
+
\left(
\begin{array}{cc}
B & -X \\
-X^* & A \\
\end{array}
\right)
=
\left(
\begin{array}{cc}
A+B & 0 \\
0 & A+B   \\
\end{array}
\right),
$$
and so
\begin{equation}\label{A}
\frac{1}{2} H
\leq
\left(
\begin{array}{cc}
\frac{A+B}{2} & 0 \\
0 & \frac{A+B}{2}  \\
\end{array}
\right).
\end{equation}
The second part of Proposition \ref{Moh-Nov} implies that

\begin{equation}\label{AA}
\left(
\begin{array}{cc}
 0& X \\
X^* & 0 \\
\end{array}
\right)
\leq
\frac{1}{2}
H.
\end{equation}
Hence, combining \eqref{A} and \eqref{AA} gives
$$
\left(
\begin{array}{cc}
0 & X \\
X^* & 0 \\
\end{array}
\right)
\leq
\left(
\begin{array}{cc}
\frac{A+B}{2} & 0 \\
0 & \frac{A+B}{2} \\
\end{array}
\right).
$$
Consequently,
$
\left(
\begin{array}{cc}
\frac{A+B}{2} & -X \\
-X^* & \frac{A+B}{2} \\
\end{array}
\right) \geq 0,
$
and then, by Proposition \ref{Moh-Nov}, we have
$
\left(
\begin{array}{cc}
\frac{A+B}{2} & X \\
X^* & \frac{A+B}{2} \\
\end{array}
\right) \geq 0.
$
A similar argument implies that
$
\left(
\begin{array}{cc}
\frac{A+B}{2}  & X^* \\
X & \frac{A+B}{2} \\
\end{array}
\right) \geq 0.
$
This proves that
$
\left(
\begin{array}{cc}
\frac{A+B}{2} & X \\
X^* & \frac{A+B}{2} \\
\end{array}
\right)
$
is PPT.
\end{proof}

In the following paragraph, we present the definition of the weighted geometric mean of two positive
matrices and then we state some of its properties.

For positive definite $X, Y \in \M$ and $t \in [0,1]$, the weighted geometric mean of $X$ and $Y$ is defined as follows
$$
X \#_{t} Y= X^{1/2}(X^{-1/2}Y X^{-1/2} )^{t}X^{1/2}.
$$
When $t=\frac{1}{2}$, we drop $t$ from the above definition, and we simply write $X \#Y$ and call it the geometric mean of $X$ and $Y$.
It is well-known that

\begin{equation}\label{A-G-mean inequality}
X \#_t Y \leq (1-t) X + t Y.
\end{equation}
See \cite[Chapter 4]{B-Book 2-2007}.

When $t=\frac{1}{2}$ an extremal property of the geometric mean of positive definite $X, Y \in \M$ is given as follows
\begin{equation}\label{Geometric mean characterization}
 X \# Y = \max \left\{ Z : Z=Z^*,
\left(
  \begin{array}{cc}
    X & Z \\
    Z & Y \\
  \end{array}
\right)\geq 0
\right\}.
\end{equation}
See \cite[Theorem 4.1.3]{B-Book 2-2007}.

For every unitarily invariant norm we have
\begin{align}\label{Norm Inq. for Geometric mean}
||X \#_t Y|| &\leq  || X^{1-t}Y^{t}||   \notag  \\
&\leq  || (1-t)X+ t Y||.
\end{align}

See \cite{Bhatia-Grover-2012}.

\section{Main results}
We start this section by the following two lemmas.
\begin{lemma}\label{TA}
If
$
\left(
\begin{array}{cc}
A_j & X \\
X^* & B_j \\
\end{array}
\right)
\geq 0 \quad (j=1,2),
$
then
$$
\left(
\begin{array}{cc}
A_1 \#_t A_2 & X \\
X^* & B_1 \#_t B_2  \\
\end{array}
\right)
\geq 0, \forall t \in [0,1].
$$
\end{lemma}
\begin{proof}
Without loss of generality we may assume that for $j=1,2$ the block
$
\left(
\begin{array}{cc}
A_j & X \\
X^* & B_j \\
\end{array}
\right)
$
is positive definite, otherwise we use the well know continuous argument. Therefore, by Schure criterion \eqref{Schure criterion}, we have
$$
X^* A_1^{-1} X  \leq B_1 \quad \text{and} \quad X^* A_2^{-1} X \leq B_2.
$$
Observe,
\begin{align}
X^* ( A_1 \#_t A_2)^{-1}) X & = X^* ( A_1^{-1} \#_t A_2^{-1} ) X  \notag \\
& = (  X^*A_1^{-1} X)  \#_t   (X^*A_2^{-1} ) X   \notag \\
& \leq B_1 \#_t   B_2 \quad   \text{(by the increasing property of means)}. \notag
\end{align}
And so
$
B_1 \#_t   B_2 \geq  X^* ( A_1 \#_t A_2)^{-1}) X .
$
This implies the result.
\end{proof}

\begin{lemma}\label{Geometric means-PPT}
If
$
\left(
\begin{array}{cc}
A & X \\
X^* & B \\
\end{array}
\right)
$
is PPT, then for every $t \in [0,1]$ the block
$
\left(
\begin{array}{cc}
A\#_t B & X \\
X^* & A \#_{1-t} B \\
\end{array}
\right)
$
is PPT.
\end{lemma}
\begin{proof}
The result follows directly from Lemma \ref{TA}, Proposition \ref{Moh-Nov-2} and the fact that $B\#_t A =A\#_{1-t} B$.
\end{proof}

Recall that the absolute value of $X \in \M$ is defined as $|X|=(X^*X)^{1/2}$.

The main result can be stated as follows.

\begin{theorem}\label{Main A result}
Let
$
\left(
\begin{array}{cc}
A & X \\
X^* & B \\
\end{array}
\right)
$
be PPT and let $X=U|X|$ be the polar decomposition of $X$. Then
$$
|X| \leq \left ( A \#_t B \right)  \#  U^*(A \#_{1-t} B)U,   \quad  \forall  t \in [0,1].
$$
\end{theorem}
\begin{proof}

Let $X=U|X|$ be the polar decomposition of $X$. Let $ W$ be the unitary defined as
$
W=
\left(
\begin{array}{cc}
U & 0 \\
0 & I \\
\end{array}
\right).
$
Since
$
\left(
\begin{array}{cc}
A & X \\
X^* & B \\
\end{array}
\right)
$
is PPT, Lemma \ref{Geometric means-PPT} implies that
$
\left(
\begin{array}{cc}
A \#_t B & X \\
X^* & A \#_{1-t} B \\
\end{array}
\right) \geq 0,
$
for every $t \in [0,1]$.
Therefore,
$$
W^*
\left(
\begin{array}{cc}
A \#_t B & X \\
X^* & A \#_{1-t} B \\
\end{array}
\right)
W
=
\left(
\begin{array}{cc}
U^*(A \#_t B)U  &  |X| \\
|X|  & A \#_{1-t} B \\
\end{array}
\right)
\geq 0.
$$
By the extremal property of the geometric mean \eqref{Geometric mean characterization} we get

$$
|X| \leq \left ( A \#_t B \right)  \# U^*(A \#_{1-t} B)U .
$$
This proves the result.
\end{proof}

\begin{corollary}\label{Main Cor result}
Let
$
\left(
\begin{array}{cc}
A & X \\
X^* & B \\
\end{array}
\right)
$
be PPT and let $X=U|X|$ be the polar decomposition of $X$. Then for some unitary $U \in \M$
$$
|X| \leq \frac{ (A \#_t B)  +  U^*(A \#_{1-t} B)U }{2}, \quad \forall t \in [0,1].
$$
In particular, 
$$
|X| \leq \frac{ (A \# B)  +  U^*(A \# B)U }{2}.
$$
\end{corollary}

We remark that the particular case $t=1/2$ of Theorem \ref{Main A result} and  Corollary \ref{Main Cor result} can be found in \cite{Pan-Fu-2021} and  \cite{Lee-2015}, respectively.

\begin{corollary}\label{norm of GM}
If
$
\left(
\begin{array}{cc}
A & X \\
X^* & B \\
\end{array}
\right)
$
is PPT, then for every unitarily invariant norm $|| \cdot ||$  and for $t \in [0,1]$
\begin{align}
||X||  & \leq ||      ( A \#_t B)  \#  U^*( A \#_{1-t} B) U    || \notag \\
& \leq ||      ( A \#_t B)^{1/2} U^*( A \#_{1-t} B)^{1/2} U    || \notag \\
& \leq \left \| \frac{  ( A \#_t B) +  U^*( A \#_{1-t} B)U }{2} \right \|  \notag \\
& \leq \frac {\left \|A \#_t B \right \| +  \left \|A \#_{1-t} B  \right \| }{2} \notag \\
& \leq  \frac{||A^{1-t}B^{t}||+ ||A^{t}B^{1-t}||}{2} \notag \\
& \leq  \frac{ ||(1-t)A+ t B|| + ||t A+ (1-t) B|| }{2}, \notag
\end{align}
for some unitary $U \in \M$.
\end{corollary}

\begin{proof}
The first inequality follows directly from Theorem \ref{Main A result}. The fourth is just the triangle inequality. 
The other inequalities follow from \eqref{Norm Inq. for Geometric mean}.
\end{proof}

In particulary, when $t=1/2$  we have the following result.

\begin{corollary}\label{norm of GM1}
If
$
\left(
\begin{array}{cc}
A & X \\
X^* & B \\
\end{array}
\right)
$
is PPT, then for every unitarily invariant norm $|| \cdot ||$  and for $t \in [0,1]$
\begin{align}
||X||  & \leq ||      ( A \# B)  \#  U^*( A \# B) U    || \notag \\
& \leq ||      ( A \# B)^{1/2} U^*( A \# B)^{1/2} U    || \notag \\
& \leq \left \| \frac{  ( A \# B) +  U^*( A \# B)U }{2} \right \|  \notag \\
& \leq \left \|A \# B \right \| \notag \\
& \leq  ||A^{1/2}B^{1/2}||  \notag \\
& \leq  \left\| \frac{A+B}{2} \right \| , \notag
\end{align}
for some unitary $U \in \M$.
\end{corollary}

If we square the inequalities in Corollary \ref{norm of GM1} and choose  the Hilbert-Schmidt norm, $|| \cdot ||_2$, we get the following result,
which is an improvement of the trace inequality  \eqref{Ming-tr}. Recall that the Hilbert-Schmidt norm is defined as $|| X ||_2^2 = tr (X^*X)$.

\begin{corollary}\label{trace-TTB}
If
$
\left(
\begin{array}{cc}
A & X \\
X^* & B \\
\end{array}
\right)
$
is PPT, then

\begin{align}
tr(X^*X) & \leq tr(A \# B)^2 \notag \\
& \leq  tr( AB). \notag \notag \\
& \leq tr \left( \frac{A+B}{2} \right)^2. \notag
\end{align}
\end{corollary}

Finally, we study the connection between the diagonal components and the real part of the off-diagonal components of the PPT block H.
Before doing so, we recall that every $X \in \M$ admits what is called the cartesian decomposition
$$
X=Re(X)+iIm(X),
$$ where $Re(X)$ and $Im(X)$ are the
Hermitian matrices defined as $Re(X)=\frac{X+X^*}{2}, Im(X)=\frac{X-X^*}{2i}$ and are known, respectively, as the real and
the imaginary parts of $X$.

\begin{theorem}\label{Realprat PPT}
Let
$
\left(
\begin{array}{cc}
A & X \\
X^* & B \\
\end{array}
\right)
$
be PPT. Then $\forall t \in [0,1]$
$$
Re(X) \leq (A\#_t B) \# (A\#_{1-t} B) \leq \frac{(A\#_t B) + (A\#_{1-t} B)}{2},
$$
and 
$$
Im(X) \leq (A\#_t B) \# (A\#_{1-t} B) \leq \frac{(A\#_t B) + (A\#_{1-t} B)}{2}.
$$
\end{theorem}

\begin{proof}
In first part, the second inequality follows from \eqref{Norm Inq. for Geometric mean}. For the second inequality, notice that, 
by Lemma \ref{Geometric means-PPT}, we have
$$
\left(
\begin{array}{cc}
A\#_t B &  X \\
X^* & A \#_{1-t} B \\
\end{array}
\right) \geq 0,
\quad  \text{and} \quad
\left(
\begin{array}{cc}
A\#_t B & X^* \\
 X &A\#_{1-t} B \\
\end{array}
\right) \geq 0,
$$
for $t \in [0,1] $.
Therefore,
$$
\left(
\begin{array}{cc}
A\#_t B  & Re(X) \\
Re(X) & A\#_{1-t} B   \\
\end{array}
\right)
=
\frac{1}{2}
\left(
\begin{array}{cc}
A\#_t B  & X \\
X^* & A\#_{1-t} B  \\
\end{array}
\right)
+
\frac{1}{2}
\left(
\begin{array}{cc}
A\#_t B & X^* \\
X &A\#_{1-t} B  \\
\end{array}
\right)
\geq 0.
$$
Therefore, by the extremal property of the geometric mean we have 

$$
Re(X) \leq  (A\#_t B ) \# (A\#_{1-t} B )
$$
This implies the first inequality.
To prove the second inequality just applying the first inequality to the block
$G=
\left(
\begin{array}{cc}
A  & -i X \\
i X^* & B  \\
\end{array}
\right).
$
Note that $G$ is PPT by Proposition \ref{Moh-Nov-3}.
\end{proof}

\begin{corollary}
Let
$
\left(
\begin{array}{cc}
A & X \\
X & B \\
\end{array}
\right) \geq 0.
$
If $X$ is Hermitian, then
$$
X \leq (A\#_t B) \# (A\#_{1-t} B) \leq \frac{(A\#_t B) + (A\#_{1-t} B)}{2}, \forall t \in [0,1].
$$
\end{corollary}

\section*{Declarations:}

\subsection*{Competing interests:}
The authors declare that they have no conflict of interest.

\subsection*{Authors' contributions:}
Not applicable.

\subsection*{Funding:}
Not applicable.

\subsection*{Availability of data and materials:}
Not applicable.

\bibliographystyle{spbasic}

\end{document}